\documentclass[11pt,reqno]{amsart}
\usepackage{amsfonts,amsmath}
\newtheorem{thm}{Theorem}[section]
\newtheorem{proposition}[thm]{Proposition}
\newtheorem{corollary}[thm]{Corollary}
\newtheorem{lemma}[thm]{Lemma}

\usepackage{graphicx}
\usepackage{epstopdf}

\vspace{4ex}
\begin{document}
\title[Deformed Macdonald-Ruijsenaars operators and super Macdonald polynomials]
{Deformed Macdonald-Ruijsenaars operators and super Macdonald polynomials}
\author{A.N. Sergeev}
\address{Department of Mathematical Sciences,
Loughborough University, Loughborough LE11 3TU, UK and Steklov Institute of Mathematics,
Fontanka 27, St. Petersburg, 191023, Russia}
\email{A.N.Sergeev@lboro.ac.uk}
\author{A.P. Veselov}
\address{Department of Mathematical Sciences,
Loughborough University, Loughborough LE11 3TU, UK  and Landau
Institute for Theoretical Physics, Moscow, Russia}
\email{A.P.Veselov@lboro.ac.uk}
\maketitle

\begin{abstract}
It is shown that the deformed Macdonald-Ruijsenaars operators can be described as the restrictions on certain affine subvarieties of the usual Macdonald-Ruijsenaars operator in infinite number of variables. The ideals of these varieties are shown to be generated by the Macdonald polynomials related to Young diagrams with special geometry. The super Macdonald polynomials and their shifted version are introduced,  the combinatorial formulas for them are given.
\end{abstract}
\section{introduction}

In this paper we investigate the properties of the deformed Macdonald-Ruijsenaars (MR) operators introduced in \cite{SV}
\begin{equation}
\label{defMR} \mathcal M_{n,m,q,t}=\frac{1}{1-q}\sum_{i=1}^n
A_{i}(T_{q,x_{i}}-1)+\frac{1}{1-t}\sum_{j=1}^m B_{j}(T_{t,y_{j}}-1),
\end{equation}
where
$$ A_{i}=\prod_{k\ne i}^{n}
\frac{(x_{i}-tx_{k})}{(x_{i}-x_{k})}\prod_{j=1}^{m}
\frac{(x_{i}-qy_{j})}{(x_{i}-y_{j})}, \quad B_{j}=\prod_{i=1}^{n}
\frac{(y_{j}-tx_{i})}{(y_{j}-x_{i})}\prod_{l\ne j}^{m}
\frac{(y_{j}-qy_{l})}{(y_{j}-y_{l})}
$$
and
$T_{q,x_{i}},T_{t,y_{j}}$ are the "shift" operators:
$$(T_{q,x_{i}}f)(x_{1},\dots,x_{i},\dots,x_{n},y_{1},\dots,y_{m})=f(x_{1},\dots,qx_{i},\dots,x_{n},y_{1},\dots,y_{m})
$$ $$
(T_{t,y_{j}}f)(x_{1},\dots,x_{n},y_{1},\dots,y_{j},\dots,y_{m})=f(x_{1},\dots,x_{n},y_{1},\dots,ty_{j},\dots,y_{m}).
$$
More precisely, we generalise the results of our paper \cite{SV2} by showing that the deformed MR operator can be described as the restriction of the usual Macdonald-Ruijsenaars operator \cite{R, Ma}
\begin{equation}
\label{MRinf}
 {\mathcal M}_{q,t}=\frac{1}{1-q}\sum_{i \geq 1}
\prod_{j\ne
i}\frac{z_{i}-t z_{j}}{z_{i}-z_{j}}\left(T_{q,z_{i}}-1\right)
\end{equation}
for infinite number of variables $z_i$ onto certain subvarieties $\Delta(n,m,q,t)$.

As well as our previous paper \cite{SV2} this work is based on the theory of
Macdonald polynomials \cite{Ma} and shifted Macdonald polynomials developed by Knop, Sahi and Okounkov \cite{KSa, K, Sa, Ok}.
The paper \cite{FJMM} by B. Feigin, Jimbo, Miwa and Mukhin was very useful for us in understanding the role of special parameters in this problem. Another important relevant work is the paper \cite{Cha} by Chalykh, who used a different technique to derive and investigate the deformed MR operator in the case  $m=1,$ which was the first case when the deformed Calogero-Moser systems were discovered (see \cite{CFV}).

The structure of the paper is following. First we review some basic facts from the theory of Macdonald polynomials and Cherednik-Dunkl operators.
The main results about deformed MR operators are proved in section 5. 
We introduce the super Macdonald polynomials as the restriction of the usual Macdonald polynomials on $\Delta(n,m,q,t).$ In section 6 we define their shifted versions and show that for any shifted super Macdonald polynomial there exists a difference operator commuting with $\mathcal M_{n,m,q,t}$  (a deformed version of Harish-Chandra homomorphism). 
In the last section we present some combinatorial formulas for the
super Macdonald polynomials and their shifted versions generalising Okounkov's result \cite{Ok}.

\section{ Symmetric functions and Macdonald polynomials}

In this section we recall some general facts about symmetric
functions and Macdonald polynomials mainly following classical Macdonald's
book \cite{Ma}. It will be convenient for us to use instead of
the parameters $q,t$ in Macdonald's notations of Macdonald
polynomials the parameters $q,t^{-1}$.

Let $P_{N}={\mathbb C}[x_{1},\dots,x_{N}]$ be the polynomial
algebra in $N$ independent variables and $\Lambda_{N}\subset
P_{N}$ be the subalgebra of symmetric polynomials.

A {\it partition} is any sequence $$
\lambda=(\lambda_{1},\lambda_{2},\dots,\lambda_{r}\,\dots) $$ of
nonnegative integers in decreasing order $$
\lambda_{1}\ge\lambda_{2}\ge\dots\ge\lambda_{r}\,\ge\dots $$
containing only finitely many nonzero terms. The number of nonzero
terms in $\lambda$ is the {\it length} of $\lambda$ denoted by
$l(\lambda)$. The sum $\mid\lambda\mid =
\lambda_{1}+\lambda_{2}+\dots$ is called the {\it weight} of
$\lambda$. The set of all partitions of weight $N$ is denoted by
${\mathcal P}_N.$

On this set there is a natural involution: in the standard
diagrammatic representation \cite{Ma} it corresponds to the
transposition (reflection in the main diagonal). The image of a
partition $\lambda$ under this involution is called the {\it
conjugate} of $\lambda$ and denoted by $\lambda'.$ This involution
will play an essential role in our paper.

Partitions can be used to label the bases in the symmetric algebra
$\Lambda_{N}.$ There are the following two standard bases in $\Lambda_{N}$,
which we are going to use:
{\it monomial symmetric polynomials} $m_{\lambda}, \lambda \in
{\mathcal P}_N,$ which are defined by $$
m_{\lambda}(x_{1},\dots,x_{N})=\sum
x_{1}^{a_{1}}x_{2}^{a_{2}}\dots x_{N}^{a_{N}}$$ summed over all
distinct permutations $a$ of $\lambda =
(\lambda_{1},\lambda_{2},\dots,\lambda_{N})$
and {\it power sums}
$$p_{\lambda}=p_{\lambda_{1}}p_{\lambda_{2}}\dots p_{\lambda_N}$$
where
$$ p_{k}=x_{1}^{k}+x_{2}^{k}+\dots + x_N^k.$$ 
It is well-known \cite{Ma} that each of these sets of functions
with $l(\lambda)\le N$ form a basis in $\Lambda_{N}.$

We will need the following infinite dimensional versions of both
$P_{N}$ and $\Lambda_{N}$. Let $M\le N$ and $\varphi_{N,M} :
P_{N}\longrightarrow P_{M}$ be the homomorphism which sends each
of $x_{M+1},\dots,x_{N}$ to zero and other $x_{i}$ to themselves.
It is clear that $\varphi_{N,M} (\Lambda_{N})=\Lambda_{M}$ so we
can consider the inverse limits in the category of graded algebras
$$ P=\lim_{\longleftarrow}
P_{N},\quad\Lambda=\lim_{\longleftarrow} \Lambda_{N}.$$ This means
that $$ P=\oplus_{r=0}^{\infty} P^{r},\quad
P^{r}=\lim_{\longleftarrow}P_{N}^{r} $$ $$
\Lambda=\oplus_{r=0}^{\infty}
\Lambda^{r},\quad\Lambda^{r}=\lim_{\longleftarrow}\Lambda_{N}^{r}
$$ where $P_{N}^r, \Lambda_{N}^r$ are the homogeneous components
of $P_{N}, \Lambda_{N}$ of degree $r$. The elements of $\Lambda$
are called {\it symmetric functions.}
Since for any partition $\lambda$  $$
\varphi_{N,M}(m_{\lambda}(x_{1},\dots,x_{N}))=m_{\lambda}(x_{1},\dots,x_{M})$$
(and similarly for the power sums) we can define the
symmetric functions $m_{\lambda}, 
p_{\lambda}.$

Another important example of symmetric functions are {\it
Macdonald  polynomials} $P_{\lambda}(x,q,t)$. We give here their definition in the form
most suitable for us.

Recall that on the set of partitions ${\mathcal P}_N$ there is the
following {\it dominance partial ordering}: we write
$\mu\le\lambda$ if for all $i \geq 1$ $$
\mu_{1}+\mu_{2}+\dots+\mu_{i}\le\lambda_{1}+\lambda_{2}+\dots+\lambda_{i}.
$$

Consider the following {\it Macdonald-Ruijsenaars operator} (MR operator)
\begin{equation}
\label{MR}
 {\mathcal M}_{q,t}^{(N)}=\frac{1}{1-q}\sum_{i=1}^N
\prod_{j\ne
i}\frac{x_{i}-tx_{j}}{x_{i}-x_{j}}\left(T_{q,x_{i}}-1\right)
\end{equation}
where $T_{q,x_{i}}$ is the shift operator
$$
\left(T_{q,x_{i}}f\right)(x_{1},\dots ,x_{i},\dots,x_{N})=
f(x_{1},\dots ,qx_{i},\dots,x_{N})
$$
This operator is related to the operators $D_{N}^{1}$ and
$E_{N}$ from Macdonald's book \cite{Ma} by the simple formulas
$$
{\mathcal
M}_{q,t}^{(N)}=\frac{t^{N-1}}{1-q}D_{N}^{1}(q,t^{-1})-\frac{1-t^{N}}{(1-q)(1-t)}=
\frac{t^{-1}}{1-q} E_{N}(q,t^{-1}).
$$
Our choice of the additional coefficient $\frac{1}{1-q}$ in formula (\ref{MR}) was motivated by the symmetric form of the deformed operator (\ref{defMR}).
We should note also that the operator (\ref{MR}) is related in a simple way to the
trigonometric version of the operator $\hat S_1$ introduced by Ruijsenaars \cite{R}.

An important property of the MR operator is its {\it
stability} under the change of $N$: the following diagram is
commutative
$$
\begin{array}{ccc}
\Lambda_{N}&\stackrel{{\mathcal
M}_{q,t}^{(N)}}{\longrightarrow}&\Lambda_{N}
\\ \downarrow
\lefteqn{\varphi_{N,M}}& &\downarrow \lefteqn{\varphi_{N,M}}\\
\Lambda_{M}&\stackrel{{\mathcal
M}_{q,t}^{(M)}}{\longrightarrow}&\Lambda_{M} \\
\end{array}
$$ (see page 321 in \cite {Ma}). This allows us to
define the MR operator ${\mathcal M}_{q,t}$ on the space of
symmetric functions $\Lambda$ as the inverse limit of ${\mathcal
M}_{q,t}^{(N)}.$

Recall \cite{Ma} that  Macdonald polynomials $P_{\lambda}(x,q,t) \in \Lambda_{N}$ are uniquely defined for generic parameters $q,t$ and any
partition $\lambda$, $l(\lambda)\le N$ by the following properties:

1) $P_{\lambda}(x,q,t)= m_{\lambda}+\sum_{\mu<\lambda}
u_{\lambda\mu}m_{\mu}$, where $u_{\lambda\mu}= u_{\lambda\mu}(q,t)\in \mathbb C$

2) $ P_{\lambda}(x,q,t)$ is an eigenfunction of the Macdonald
operator ${\mathcal M}_{q,t}^{(N)}.$

Indeed the operator ${\mathcal
M}_{q,t}^{(N)}$ has an upper triangular matrix in the monomial
basis $m_{\mu}$: $$ {\mathcal
M}_{q,t}^{(N)}(m_{\lambda})=\sum_{\mu\le\lambda}
c_{\lambda\mu}m_{\mu},$$ where the coefficients $c_{\lambda\mu}$
can be described explicitly (see \cite{Ma}, page 321). In particular
$$
c_{\lambda,\lambda}=\frac{1}{1-q}\sum_{i=1}^N\left(q^{\lambda_{i}}-1\right)t^{i-1}.
$$


For generic parameters $q,\,t$ the coefficients $c_{\lambda,\lambda}\ne c_{\mu,\mu}$
for all $\lambda\ne\mu$ with $|\lambda| = |\mu|,$ so the operator ${\mathcal M}_{q,t}^{(N)}$
is diagonalisable.

We should note that the coefficient $u_{\lambda\mu}$ are rational functions of $q$ and $t,$
which have the singularities only if $q^a = t^b$ for some non-negative integers $a, b$ (not equal to zero simultaneously) \cite{Ma}. Such parameters are called {\it special}, the Macdonald polynomials are well-defined for all non-special values of parameters $q,t$.

From the stability of the Macdonald operators it follows that
$$\varphi_{N,M}(P_{\lambda}(x_{1},\dots,x_{N}))=P_{\lambda}(x_{1},\dots,x_{M})
$$ so we have correctly defined Macdonald symmetric functions
$P_{\lambda}(x,q,t)\in \Lambda$ which are the eigenfunctions of
the Macdonald operator ${\mathcal M}_{q,t}^{(N)}.$

\section{Shifted symmetric functions and shifted Macdonald polynomials.}

We discuss now the so-called {\it shifted Macdonald polynomials}
investigated by Knop, Sahi and Okounkov
\cite{KSa, K, Sa, Ok}.

Let us denote by $\Lambda _{N,t} $ the algebra of polynomials
$f(x_{1},\dots,x_{N})$ which are symmetric in the "shifted"
variables $ x_{i}t^{i-1}$. This algebra has the filtration by the
degree of polynomials: $$ (\Lambda_{N,t})_{0}\subset
(\Lambda_{N,t})_{1}\subset \dots\subset
(\Lambda_{N,t})_{r}\subset\dots $$ We have the following shifted
analog of power sums:
\begin{equation}
\label{ps} p^*_{r}(x_{1},\dots,x_{N},t)=\sum_{i=1}^N
\left(x_{i}^r-1\right)t^{r(i-1)}
\end{equation}
The polynomials $$
p^*_{\lambda}(x,t)=p^*_{\lambda_{1}}(x,t)p^*_{\lambda_{2}}(x,t)\dots,
$$ where
$\lambda=(\lambda_{1},\lambda_{2},\dots,\lambda_{r}\,\dots)$ are
partitions of length $l(\lambda) \leq N,$ form a basis in $\Lambda _{N,t}.$  
They are stable in the following sense.
Let $M\le N$ and $\varphi_{N,M}^{*} :
P_{N}\longrightarrow P_{M}$ be the homomorphism which sends each
of $x_{M+1},\dots,x_{N}$ to 1 and leaving the remaining $x_{i}$ the same.
Then $\varphi_{N,M}^{*}(p^*_{\lambda}(x_1, \dots, x_N))= p^*_{\lambda}(x_1, \dots, x_M).$
 Therefore $\varphi_{N,M}^{*}
(\Lambda_{N,t})=\Lambda_{M,t}$ and one can consider the inverse limit
$$\Lambda_{t} =\lim_{\longleftarrow} \Lambda_{N,t}$$ in the
category of filtered algebras: $$
\Lambda_{t}=\bigcup_{r=0}^{\infty}
(\Lambda_{t})_{r},\quad(\Lambda_{t})_{r}=\lim_{\longleftarrow}(\Lambda_{N,t})_{r}.
$$ The algebra $\Lambda_{t}$ is called the algebra of {\it
shifted symmetric functions} \cite{Sa,Ok}. Let us introduce the following
function on the set of partitions:
\begin{equation}
\label{H}
 H(\lambda,q,t)=t^{n(\lambda^{\prime})}q^{n(\lambda)}\prod_{s\in\lambda}\left(
 q^{a(s)+1}-t^{l(s)}\right).
 \end{equation}
Here  $a(s)$ and  $l(s)$ are  "arm" and "leg" lengths respectively of a box $s = (i,j) \in \lambda$, which are defined
$$
a(s)=\lambda_{i}-j,\quad\l(s)=\lambda_{j}^{\prime}-i
$$
and
$$
n(\lambda)=\sum_{i\geq 1}(i-1)\lambda_{i}.
$$

Recall (see \cite{K, Sa, Ok}) that  the {\it shifted Macdonald polynomial} $P_{\lambda}^{*}(x,q,t) \in \Lambda_{t}$ is a unique shifted symmetric function of degree $\deg
P_{\lambda}= |\lambda|$ satisfying the following property:
 $$
P_{\lambda}^{*}(q^{\lambda},q,t)= H(\lambda,q,t)$$
 and $P_{\lambda}^{*}(q^{\mu},q,t) = 0$ unless $\lambda \subseteq \mu$ (Extra Vanishing Condition).
Here and later throughout the paper by $P(q^{\lambda})$ for a partition
$\lambda=(\lambda_{1},\dots,\lambda_{N})$ we mean
$P(q^{\lambda_1},\dots,q^{\lambda_{N}}, 1, 1, \dots).$



We will need the following duality property of the shifted
Macdonald polynomials proved by Okounkov \cite{Ok} 
\begin{equation}
\label{dual}
P_{\lambda}^{*}(q^{\mu},q,t)=\frac{H(\lambda,q,t)}{H(\lambda',t,q)}
P_{\lambda'}^{*}(t^{\mu'},t,q).
\end{equation}
To show this consider the
following {\it conjugation homomorphism} (cf. \cite{SV2}):
\begin{equation}
\label{omega}
\left[(\omega_{q,t}^*(f))\right](t^{\lambda})=f(q^{\lambda^{\prime}})
\end{equation}
We claim that the conjugation homomorphism maps the
algebra of shifted symmetric functions $\Lambda_{t}$ into the
algebra $\Lambda_{q}.$
Indeed computing the sum
$$
\sum_{(i,j)\in\lambda}q^{i-1}t^{j-1}
$$
first along columns and then along the rows we come to the following equality 
\begin{equation}
\label{conjugate}
\frac{1}{1-q}\sum_{j\ge
1}\left(q^{\lambda_{j}^{\prime}}-1\right)t^{j-1}=
\frac{1}{1-t}\sum_{i\ge 1}\left(t^{\lambda_{i}}-1\right)q^{i-1},
\end{equation}
which is equivalent to 
$$
(\omega_{q,t}^*(p_{r}^*(x,t))=\frac{1-q^r}{1-t^r}p_{r}^*(x,q)
$$
with $r=1$. Replacing $q$ by
$q^{r}$ and $t$ by $t^r$ we have this formula for all $r.$ Now the claim follows from the fact that $p^*_{\lambda}(x,t)$ generate the algebra. 

Combining this
with the definition of the shifted Macdonald  polynomials we have
the duality property (\ref{dual}).

\section{Cherednik - Dunkl operators and Harish-Chandra homomorphism}

In this section we present the basic facts about Cherednik-Dunkl
operators. For more details we refer to \cite{Ch,KN}.

Consider the operators $T_{i}, i=1,\dots,N-1$
$$
T_{i}=1+\frac{x_{i}-tx_{i+1}}{x_{i}-x_{i+1}}(\sigma_{ii+1}-1)
$$
and
$$
\omega={\sigma}_{NN-1}\sigma_{N-1N-2}\dots\sigma_{21}T_{q,x_{1}},
$$
where $\sigma_{ij}$ is acting on the function
$f(x_{1},\dots,x_{N})$ by permutation of the $i$-th and $j$-th
coordinates.
 By {\it Cherednik-Dunkl operators} we will mean the
following difference operators
\begin{equation}
\label{CD}
 D_{i,N} = t^{1-N}T_{i}\dots T_{N-1}\omega T_{1}^{-1}\dots
 T_{i-1}^{-1},\quad i=1,\dots N.
\end{equation}

The first important property of the Cherednik-Dunkl operators is
that they commute with each other: $$[D_{i,N}, D_{j,N}] = 0.$$
This means that one can substitute them in any polynomial $P$ in
$N$ variables without ordering problems.

The second property is that if one does this for a shifted
symmetric polynomial $g \in \Lambda_{N,t}$ then the corresponding
operator $g(D_{1,N} \dots D_{N,N})$ leaves the algebra of
symmetric polynomials $\Lambda_N$ invariant: $$ g(D_{1,N} \dots
D_{N,N}): \Lambda_N \rightarrow \Lambda_N.$$ The restriction of
the operator $g(D_{1,N} \dots D_{N,N})$ on the algebra $\Lambda_N$
is given by some difference operator, which we will denote as
${\mathcal D}^g_{N,q,t}.$ 

One can check that if we apply this operation to the shifted
power sum $p^*_{1}(x_{1},\dots,x_{N},t)=\sum_{i=1}^N
\left(x_{i}-1\right)t^{i-1}$ we arrive (up to a factor $(1-q)^{-1}$) at Macdonald-Ruijsenaars operator
(\ref{MR}). Thus the operators ${\mathcal D}^g_{N, q,t}$ can be considered as the integrals of the corresponding quantum system, which is equivalent to the relativistic Calogero-Moser system introduced by Ruijsenaars \cite{R}. 

The Macdonald polynomials are the
joint eigenfunctions of all these operators: if
$P_{\lambda}(x,q,t)$ is the Macdonald polynomial corresponding a
partition $\lambda$ of weight $N$ then
\begin{equation} \label{LP}
{\mathcal D} ^g_{N,t} P_{\lambda}(x,q,t)=
g(q^{\lambda_{1}},q^{\lambda_{2}},\dots,q^{\lambda_{N}})P_{\lambda}(x,q,t)
\end{equation}

This allows us to define a homomorphism (which is actually a
monomorphism) $\chi: f \rightarrow {\mathcal D}^f_{N,t}$ from the
algebra $\Lambda_{t}$ to the algebra of difference operators. Let
us denote by ${\mathcal D}(N,t)$ the image of $\chi$. The inverse
homomorphism $$ \chi^{-1} : {\mathcal D}(N,t)\longrightarrow
\Lambda _{N,t} $$ is called the {\it Harish-Chandra isomorphism.}
It can be defined by the action on the Macdonald polynomials: the
image of ${\mathcal D} \in {\mathcal D}(N,t)$ is a polynomial $f =
f_{\mathcal D} \in \Lambda _{N,t}$ such that
$${\mathcal D} P_{\lambda}(x,q,t)=
f(q^{\lambda})P_{\lambda}(x,q,t).$$

One can check that the Cherednik-Dunkl operators are stable: the
diagram
 $$
\begin{array}{ccc}
P_{N}&\stackrel{D_{i,N}}{\longrightarrow}&P_{N} \\ \downarrow
\lefteqn{\varphi_{N,M}}& &\downarrow \lefteqn{\varphi_{N,M}}\\
P_{M}&\stackrel{D_{i,M}}{\longrightarrow}&P_{M} \\
\end{array}
$$ is commutative for all $M\le N$ and $i\ge 1$. Similarly for any
$f\in \Lambda_{N,t}$ and $g=\varphi_{N,M}^*(f), M \le N$ the
following diagram is commutative: $$
\begin{array}{ccc}
\Lambda_{N}&\stackrel{{\mathcal D}^f_{N,
t}}{\longrightarrow}&\Lambda_{N} \\ \downarrow
\lefteqn{\varphi_{N,M}}& &\downarrow \lefteqn{\varphi_{N,M}}\\
\Lambda_{M}&\stackrel{{\mathcal
D}^g_{M,t}}{\longrightarrow}&\Lambda_{M} \\
\end{array}
$$ This allows us to define for any shifted symmetric function
$f\in \Lambda_{t}$ a difference operator $${\mathcal D}^{f}_{t} :
\Lambda \longrightarrow \Lambda$$ and the infinite dimensional
version of the homomorphism $\chi.$ We will denote by ${\mathcal
D} (t)$ the image of this homomorphism.  The inverse
(Harish-Chandra) homomorphism $\chi^{-1}: {\mathcal D} (t)
\longrightarrow \Lambda_{t}$ can be described by the relation
$${\mathcal D}^{f}_{t} P_{\lambda}(x, q,t) = f(q^{\lambda})
P_{\lambda}(x, q,t),$$ where now $f \in \Lambda_{t}$ and $P_{\lambda}(x, q,t)$ are Macdonald polynomials.

\section{Deformed Macdonald-Ruijsenaars operator as a restriction}

The following algebra $\Lambda_{n,m,q,t}$ will play a central role
in our construction.
Let $P_{n,m}={\mathbb C}[x_{1},\dots,x_{n},y_{1},\dots,y_{m}]$ be
the polynomial algebra in $n+m$ independent variables. Then
$\Lambda_{n,m,q,t}\subset P_{n,m}$ is the subalgebra consisting of
polynomials which are symmetric in $x_{1},\dots,x_{n}$ and
$y_{1},\dots,y_{m}$ separately and satisfy the conditions
\begin{equation}
\label{cond0} T_{q,x_{i}}(f)=T_{t,y_{j}}(f)
\end{equation}
on each hyperplane  $x_{i}-y_{j}=0$ for all $i=1,\dots,n$ and
$j=1,\dots,m.$ 

Assume from now on that $t$ and $q$ {\it are not roots of unity} and consider the following {\it deformed Newton sums}
\begin{equation}
\label{defNewton}
 p_{r}(x,y, q,t)= \sum_{i=1}^n
{x_{i}^r}+\frac{1-q^r}{1-t^r}\sum_{j=1}^m {y_{j}^r},
\end{equation}
which obviously belong to $\Lambda_{n,m, q,t}$ for all
nonnegative integers $r$.

We will prove later that if  the parameters $q,t$ are non-special, then $\Lambda_{n,m; q,t}$ is generated by the deformed Newton sums $ p_{r}(x,y, q,t)$ (see Theorem \ref{defN} below), but now
let us start with the following result.

\begin{thm} \label{fingen}
The algebra $\Lambda_{n,m, q,t}$ is finitely
generated if and only if $t^i q^j \neq 1$ for all $1\le i\le n,\: 1\le j\le m.$
\end{thm}

\begin{proof}
Consider the subalgebra $P(k) = {\bf C}[p_1, ... , p_{n+m}]$
generated by the first $n+m$ deformed Newton sums
(\ref{defNewton}). We need the following result about common zeros
of these polynomials (cf. Proposition 4 in \cite{SV}).

\begin{lemma}
The system $$ \left\{
\begin{array}{rcl}
x_{1}+x_{2}+\dots+x_{n}+\frac{1-q}{1-t}(x_{n+1}+x_{n+2}+\dots+x_{n+m})=0\\
x_{1}^2+x_{2}^2+\dots+x_{n}^2+\frac{1-q^2}{1-t^2}(x_{n+1}^2+x_{n+2}^2+\dots+x_{n+m}^2)=0\\
\cdots\\
x_{1}^{n+m}+x_{2}^{n+m}+\dots+x_{n}^{n+m}+\frac{1-q^{n+m}}{1-t^{n+m}}(x_{n+1}^{n+m}+x_{n+2}^{n+m}+\dots+x_{n+m}^{n+m})=0\\
\end{array}
\right. $$ 
has a non-zero solution in ${\bf C}^{n+m}$ if and only if
$t^{i}q^{j}=1$ for some $1\le i\le n,\: 1\le j\le m$.
\end{lemma}

\begin{proof}
Let us multiply the $k$-th equation by $1-t^k$ and rewrite it as
$$x_1^k+\dots +x_n^k+x_{n+1}^k+\dots+x_{n+m}^k = (t x_1)^k+\dots +(t x_n)^k+(q x_{n+1})^k+\dots+(q x_{n+m})^k.$$
Since this is true for all $k=1,\dots, n+m$ this means that the set $$x_{1},\dots, x_{n}, x_{n+1}, \dots, x_{n+m}$$ coincides up to a permutation with the set $$t x_{1},\dots, t x_{n}, q x_{n+1}, \dots, q x_{n+m}.$$

Let us consider only nonzero elements $x_{i}, i\in S\subset [1,\dots,n]$ and nonzero elements $x_{n+j}, j\in T\subset [1,\dots,m]$ . Therefore
$$
\prod_{i\in S}x_{i}\prod_{j\in T}x_{n+j}=t^{|S|}q^{|T|}\prod_{i\in S}x_{i}\prod_{j\in T}x_{n+j}
$$
Therefore $t^{|S|}q^{|T|}=1$. 

Conversely suppose that $t^iq^j=1$ for some $1\le i\le n,\: 1\le j\le m$ and consider
$$
x_{1}=1,\, x_{2}=t^{-1},\dots,x_{i}=t^{1-i},\quad x_{n+1}=q^{-1}t^{1-i},\, \dots, x_{n+j}=q^{-j}t^{1-i}
$$
with other variables to be zero. Then it is easy to verify that it is a solution of the system.
\end{proof}

From lemma it follows that if  $t^{i}q^{j}\neq1$ for all $1\le i\le n,\: 1\le j\le m$ the
algebra of all polynomials on $V$ is a finitely generated module
over subalgebra $P(k).$ By a general result from commutative
algebra  \cite{Atiyah-Macdonald}
this implies that $\Lambda_{n,m; q,t}$ is finitely generated.

Conversely, assume that $t^{i}q^{j} = 1$ for some $1\le i\le n,\: 1\le j\le m$
and consider the following homomorphism
$$\Phi_{i,j}: \Lambda_{n,m; q,t} \rightarrow \mathbb C[u,v]$$ sending a polynomial
$f(x_1, \dots, x_n,\, y_1, \dots, y_m)$ into
$$\phi(u,v) = f (u, t^{-1}u, \dots, t^{1-i}u,0,\dots,0, \, q^{j-1}t v, q^{j-2}t v, \dots, t v, 0, \dots, 0).$$
One can check that the image $\phi$ of any $f \in \Lambda_{n,m; q,t}$ satisfies the condition
$\phi (u, u) = \phi(qu, qu)$ and therefore $\phi(u,u) = const$  since $q$ is not root of unity. Moreover one can show that any such function $\phi$ belongs to the image of $\Phi_{i,j}.$
The corresponding algebra consists of the polynomials of the form $\phi = (u-v)p(u,v) + c,$ which is not finitely generated (c.f. \cite{SV}, p. 274). This completes the proof of the theorem.
\end{proof}


Let us assume from now on that $q,t$ are generic.
Since algebra $\Lambda_{n,m, q,t}$ is finitely generated 
we can introduce an affine algebraic variety $$\Delta_{n,m, q,t} = Spec \,
\Lambda_{n,m, q,t}.$$ 

Consider the following embedding of $\Delta_{n,m; q,t}$ into infinite-dimensional Macdonald variety ${\mathcal M} = Spec \, \Lambda$ (cf. \cite{SV2}).
Recall that  $\Lambda$ is the algebra of symmetric
functions in infinite number of variables $z_{1},z_{2},\dots$,
which is freely generated by the powers sums $p_{r}(z)=z_{1}^r+z_{2}^r+\dots.$
Consider the following homomorphism $\varphi$ from $\Lambda$ to
$\Lambda_{n,m,q,t}$ uniquely determined by the relations $$ \varphi (p_{r}(z))=
p_{r}(x,y,q,t).
$$ 
For generic $q,t$ this homomorphism is surjective and thus defines an embedding
$\phi: \Delta_{n,m, q,t} \rightarrow  {\mathcal M}.$

We are going to show that the deformed MR operator (\ref{defMR}) is
the restriction of the usual MR operator on ${\mathcal M}$ onto the subvariety
$\Delta_{n,m, q,t}.$

We start with the following modification of Proposition 2.8 from paper \cite{Cha} by Chalykh.

\begin{proposition}\label{cha}
 The deformed MR operator preserves
the algebra $\Lambda_{n,m,q,t}$:
\begin{equation} \label{invar}
{\mathcal M}_{n,m,q,t}: \Lambda_{n,m,q,t} \rightarrow
\Lambda_{n,m,q,t}.
\end{equation}
\end{proposition}

\begin{proof}  Let $f\in\Lambda_{n,m,q,t}$ and
$g={\mathcal M}_{n,m,q,t}(f)$. Then we have
$$
g=\sum_{i=1}^n\frac{A_{i}}{1-q}f_{i}+\sum_{j=1}^m\frac{B_{j}}{1-t}f_{\bar
j },
$$
where $A_{i}, \, B_j$ are the same as in (\ref{defMR}) and 
$$
f_{i}=T_{q,x_{i}}(f)-f,\quad f_{\bar j}=T_{t,y_{j}}(f)-f,\quad
i=1,\dots,n,\,j=1,\dots,m.
$$
Let us prove first that $g$ is a polynomial. It is clear that $g$
is a rational function, which is symmetric in $x_{1},\dots,x_{n}$
and $y_{1},\dots,y_{m}$, so it is enough to prove that $g$ has no
singularities of type
$(x_{1}-x_{2})^{-1}$,$(y_{1}-y_{2})^{-1}$,$(x_{1}-y_{1})^{-1}$.

Let us represent $g$ in the form
$$
g=\frac{1}{x_{1}-x_{2}}\left((x_{1}-x_{2})\frac{A_{1}}{1-q}f_{1}+(x_{1}-x_{2})
\frac{A_{2}}{1-q}f_{2}\right)+g_{1},
$$
where $g_{1}$ is a rational function without singularities on the hyperplane
$x_{1} = x_{2}$. But  on this hyperplane we have
$$
(x_{1}-x_{2})\frac{A_{1}}{1-q}f_{1}+(x_{1}-x_{2})
\frac{A_{2}}{1-q}f_{2}=\left((x_{1}-x_{2})\frac{A_{1}}{1-q}+(x_{1}-x_{2})
\frac{A_{2}}{1-q}\right)f_{1}=0,
$$
so we see that $g$ has no poles when
$x_{1} = x_{2}$. Similarly there are no
singularities of type $(y_{1}-y_{2})^{-1}$. Now write $g$ in the form
$$
g=\frac{1}{x_{1}-y_{1}}\left((x_{1}-y_{1})\frac{A_{1}}{1-q}h_{1}+(x_{1}-y_{1})
\frac{B_{1}}{1-t}f_{\bar 1}\right)+g_{2},
$$
where $g_{2}$ is a rational function without poles when
$x_{1} = y_{1}$. On the hyperplane
$x_{1} = y_{1}$ we have
$$
(x_{1}-y_{1})\frac{A_{1}}{1-q}+(x_{1}-y_{1}) \frac{B_{1}}{1-t}=0,
$$
therefore
$$
(x_{1}-y_{1})\frac{A_{1}}{1-q}f_{1}+(x_{1}-y_{1})
\frac{B_{1}}{1-t}f_{\bar 1}=\left((x_{1}-y_{1})\frac{A_{1}}{1-q}+(x_{1}-y_{1})
\frac{B_{1}}{1-t}\right)f_{1}=0.
$$
We have used here that when $x_{1} = y_{1}$  we have $f_1 = f_{\bar 1}$ from the definition of algebra
$\Lambda_{n,m,q,t}$.
Thus we have proved that $g$ is a polynomial. 

Now let us prove that $g \in
\Lambda_{n,m,q,t}$. On the hyperplane $x_{1}=y_{1}$ we have the
following equalities:
$$
\left(T_{q,x_{1}}-T_{t,y_{1}}\right)A_{i}=0,\quad i\ne 1,\quad
T_{q,x_{1}}A_{1}=0,
$$
$$
\left(T_{q,x_{1}}-T_{t,y_{1}}\right)B_{j}=0,\quad j\ne 1 \quad
T_{t,y_{1}}B_{1}=0,
$$
$$
T_{t,y_{1}}\frac{A_{1}}{1-q}=T_{q,x_{1}}\frac{B_{1}}{1-t},
$$
and therefore
$$
\left(T_{q,x_{1}}-T_{t,y_{1}}\right)g=
\left(T_{q,x_{1}}-T_{t,y_{1}}\right)\left(\frac{A_{1}}{1-q}f_{1}+
\frac{B_{1}}{1-t}f_{\bar 1}\right)=
$$
$$
T_{q,x_{1}}\left(\frac{B_{1}}{1-t}f_{\bar 1}\right)-
T_{t,y_{1}}\left(\frac{A_{1}}{1-q}f_{1}\right)=
T_{q,x_{1}}\left(\frac{B_{1}}{1-t}\right)\left(T_{q,x_{1}}T_{t,y_{1}}(f)-
T_{q,x_{1}}f\right)-
$$
$$
T_{t,y_{1}}\left(\frac{A_{1}}{1-q}\right)\left(T_{t,y_{1}}T_{q,x_{1}}(f)-
T_{t,y_{1}}f\right)=0
$$
since $f \in \Lambda_{n,m,q,t}.$
\end{proof}

Now we are ready to formulate our central result. Let ${\mathcal
M}_{q,t}$ be the usual Macdonald  operator in infinite dimensions.

\begin{thm}
 The following diagram is commutative for
all values of the parameters $q,t$:
\begin{equation} \label{commdia}
\begin{array}{ccc}
\Lambda&\stackrel{{\mathcal M}_{q,t}}{\longrightarrow}&\Lambda
\\ \downarrow \lefteqn{\varphi}& &\downarrow \lefteqn{\varphi}\\
\Lambda_{n,m,q,t}&\stackrel{{\mathcal
M}_{n,m,q,t}}{\longrightarrow}&\Lambda_{n,m,q,t} \\
\end{array}
\end{equation}
In other words, the deformed MR operator (\ref{defMR}) is
the restriction of the operator ${\mathcal
M}_{q,t}$
onto the subvariety
$\Delta_{n,m, q,t} \subset \mathcal M.$
\end{thm}

\begin{proof}
Let us introduce the following function
$\Pi \in \Lambda[[w_1, \dots, w_N]]$ which plays an important role
in the theory of Macdonald polynomials (see \cite{Ma}):
$$
\Pi = \prod_{j=1}^{N}\prod_{r = 0}^{\infty} \prod_{i=1} ^{\infty}
\frac{1-t^{-1}q^rz_{i}w_{j}}{1-q^rz_{i}w_{j}}
$$

\begin{lemma}\label{Pi}
The function $\Pi$ satisfies the following
properties:
\begin{enumerate}
\item[(i)]
\begin{equation}
\label{lemma2.1} {\mathcal M}^z_{q,t} \Pi={\mathcal M}^w_{q,t}
\Pi,
\end{equation} {\it where index $z$ (resp. $w$) indicates the
action of the Macdonald  operator ${\mathcal M}_{q,t}$ on $z$
(resp. $w$) variables}
\item[(ii)]
\begin{equation}
\label{lemma2.2} \varphi(\Pi)
=\prod_{r=0}^{\infty}\prod_{i=1}^n\prod_{l=1}^N
\frac{1-t^{-1}q^rx_{i}w_{l}}{1-q^rx_{i}w_{l}}
\prod_{j=1}^m\prod_{l}(1-t^{-1}y_{j}w_{l})
\end{equation}
\item[(iii)]
\begin{equation}
\label{lemma2.3} \varphi({\mathcal M}^z_{q,t} \Pi)={\mathcal
M}_{n,m,q,t} \varphi(\Pi)
\end{equation}
\end{enumerate}
\end{lemma}

\begin{proof} The first part is well known, see Macdonald \cite{Ma}, formula (3.12) in Chapter 6.

To prove the second one we note that since
$\varphi$ is a homomorphism it is enough to consider the case
$N=1$ when we have only one variable $w.$ 
Define the following automorphism 
$\sigma_{q,t}: \Lambda \rightarrow \Lambda$ by
\begin{equation}
\label{auto}
\sigma_{q,t}(p_{r}(y))=\frac{1-q^r}{1-t^r}p_{r}(y).
\end{equation}
We have (cf. \cite{SV2})
$$
\varphi(\Pi) = \prod_{r=0}^{\infty}\prod_{i=1}^n
\frac{1-t^{-1}q^rx_{i}w}{1-q^rx_{i}w} \sigma_{q,t}
\left(\prod_{r=0}^{\infty}\prod_{j\ge 1}
\frac{1-t^{-1}q^ry_{j}w}{1-q^ry_{j}w}\right)
$$
$$
= \prod_{r=0}^{\infty}\prod_{i=1}^n
\frac{1-t^{-1}q^rx_{i}w}{1-q^rx_{i}w} \sigma_{q,t}
\left(\exp\log\prod_{r=0}^{\infty}\prod_{j\ge 1}
\frac{1-t^{-1}q^ry_{j}w}{1-q^ry_{j}w}\right)
$$
$$
= \prod_{r=0}^{\infty}\prod_{i=1}^n
\frac{1-t^{-1}q^rx_{i}w}{1-q^rx_{i}w} \exp\sigma_{q,t}
\left(\log\prod_{r=0}^{\infty}\prod_{j\ge 1}
\frac{1-t^{-1}q^ry_{j}w}{1-q^ry_{j}w}\right)
$$
$$
=\prod_{r=0}^{\infty}\prod_{i=1}^n
\frac{1-t^{-1}q^rx_{i}w}{1-q^rx_{i}w} \exp\sigma_{q,t}
\left(\sum_{r=0}^{\infty}\sum_{j\ge 1}
\log(1-t^{-1}q^ry_{j}w)-\log(1-q^ry_{j}w)\right)
$$
$$
=\prod_{r=0}^{\infty}\prod_{i=1}^n
\frac{1-t^{-1}q^rx_{i}w}{1-q^rx_{i}w} \exp\sigma_{q,t}
\left(\sum_{r=0}^{\infty}\sum_{j\ge 1} \sum_{s\ge 1}
\frac{q^{rs}y_{j}^sw^s-t^{-s}q^{rs}y_{j}^sw^s}{s}\right)
$$
$$
=\prod_{r=0}^{\infty}\prod_{i=1}^n
\frac{1-t^{-1}q^rx_{i}w}{1-q^rx_{i}w} \exp\sigma_{q,t}
\left(\sum_{s\ge 1} \frac{1-t^{-s}}{1-q^s}p_{s}(y)\frac{w^s}{s}\right)
$$
$$
=\prod_{r=0}^{\infty}\prod_{i=1}^n
\frac{1-t^{-1}q^rx_{i}w}{1-q^rx_{i}w} \exp \left(\sum_{s\ge 1}
\frac{1-t^{-s}}{1-t^s}p_{s}(y)\frac{w^s}{s}\right)
$$
$$
=\prod_{r=0}^{\infty}\prod_{i=1}^n
\frac{1-t^{-1}q^rx_{i}w}{1-q^rx_{i}w} \exp \left(-\sum_{s\ge 1}
t^{-s}p_{s}(y)\frac{w^s}{s}\right)
$$
$$
=\prod_{r=0}^{\infty}\prod_{i=1}^n
\frac{1-t^{-1}q^rx_{i}w}{1-q^rx_{i}w}\prod_{j=1}^m(1-t^{-1}y_{j}w)
$$
This proves the second
part of the Lemma.

Let us prove $(iii)$. It is easy to check the following equalities
$$
\varphi(\Pi)^{-1}T_{q, w_{l}}\varphi(\Pi)=\prod_{i=1}^n
\frac{1-x_{i}w_{l}}{1-t^{-1}x_{i}w_{l}}\prod_{j=1}^m
\frac{1-t^{-1}q y_{j}w_{l}}{1-t^{-1}y_{j}w_{l}},
$$
$$
\varphi(\Pi)^{-1}T_{q,x_{i}}\varphi(\Pi)=\prod_{l=1}^N
\frac{1-x_{i}w_{l}}{1-t^{-1}x_{i}w_{l}},
$$
$$
\varphi(\Pi)^{-1}T_{q,y_{j}}\varphi(\Pi)=\prod_{l=1}^N
\frac{1-y_{j}w_{l}}{1-t^{-1}y_{j}w_{l}}.
$$
Now we see that $(iii)$ is equivalent to the following
equality
\begin{equation}
\label{eq}
\sum_{l=1}^N C_l \left(\prod_{i=1}^n
\frac{1-x_{i}w_{l}}{1-t^{-1}x_{i}w_{l}}\prod_{j=1}^m
\frac{1-t^{-1}q y_{j}w_{l}}{1-t^{-1}y_{j}w_{l}}-1\right)
\end{equation}
$$
=\sum_{i=1}^n A_i \left(\prod_{l=1}^N
\frac{1-x_{i}w_{l}}{1-t^{-1}x_{i}w_{l}}-1\right)
+\frac{1-q}{1-t}\sum_{j=1}^m B_j \left(\prod_{l=1}^N
\frac{1-t^{-1}q y_{j}w_{l}}{1-t^{-1}y_{j}w_{l}}-1\right),
$$
where
$$C_l = \prod_{k\ne
l}^N \frac{w_{l}-tw_{k}}{w_{l}-w_{k}},\,\,
A_i = \prod_{s\ne i}^n\frac{x_{i}-tx_{s}}{x_{i}-x_{s}}
\prod_{j=1}^m\frac{x_{i}-qy_{j}}{x_{i}-y_{j}},\,\,
B_j = \prod_{r\ne
j}^m\frac{y_{j}-qy_{r}}{y_{j}-y_{r}}
\prod_{i=1}^n\frac{y_{j}-tx_{i}}{y_{j}-x_{i}} .$$
From the partial fraction decomposition we have
$$
\prod_{i=1}^n \frac{1-x_{i}w_l}{1-t^{-1}x_{i}w_l}\prod_{j=1}^m
\frac{1-t^{-1}q y_{j}w_l}{1-t^{-1}y_{j}w_l}-1
$$
$$
=t^nq^m -1 + (1-t) \sum_{i=1}^n \frac{A_i}{1-t^{-1}x_i w_l}
+(1-q) \sum_{j=1}^m \frac{B_j}{1-t^{-1}y_j w_l},
$$
$$
\prod_{l=1}^N \frac{1-x_i w_l}{1-t^{-1}x_i w_l}-1=  t^N -1 + (1-t) \sum_{l=1}^N \frac{C_l}{1-t^{-1}x_i w_l},
$$
$$
\prod_{l=1}^N \frac{1-y_j w_l}{1-t^{-1}y_j w_l}-1=  t^N -1 + (1-t) \sum_{l=1}^N \frac{C_l}{1-t^{-1}y_j w_l}.
$$
Substituting these identities into relation (\ref{eq}) we reduce it to the following equality
$$
(t^nq^m-1) \sum_{l=1}^N C_l = (t^N-1)  \sum_{i=1}^n A_i +  \frac{1-q}{1-t} (t^N-1)  \sum_{j=1}^m B_j,$$
which follows from the identities
$$
\sum_{l=1}^N C_l= \frac{t^N-1}{t-1}, \quad \sum_{i=1}^n A_i +  \frac{1-q}{1-t} \sum_{j=1}^m B_j = \frac{t^n q^m -1}{t-1}.$$
 Lemma is proved.
\end{proof}

To complete the proof of the Theorem we need Macdonald's result that the
coefficients $g_{\lambda}(z,q,t)$ in the expansion of the function
$$
\Pi = \prod_{j=1}^{N}\prod_{r = 0}^{\infty} \prod_{i=1} ^{\infty}
\frac{1-t^{-1}q^rz_{i}w_{j}}{1-q^rz_{i}w_{j}} =
\sum_{\lambda}g_{\lambda}(z,q,t)m_{\lambda}(w)
$$
linearly generate $\Lambda$ when we increase the number of variables $w$ (see \cite{Ma}, VI, (2.10)).
\end{proof}


Let us introduce the set of partitions $H_{n,m},$ which consists
of the partitions $\lambda$ such that $\lambda_{n+1} \leq m$ or,
in other words, whose diagrams are contained in the {\it fat
$(n,m)$ - hook} (see Fig.1). Its complement we will denote as
$\bar H_{n,m}.$

\begin{center}
\includegraphics[width=7cm]{Fig1}
\end{center}

\begin{thm} 
\label{ker}
If $q,t$ are non-special then $Ker \varphi$ is
spanned by the Macdonald polynomials $P_{\lambda}(z,q,t)$
corresponding to the partitions which are not contained in the fat
$(n,m)$-hook.
\end{thm}

\begin{proof} 
Consider the automorphism 
$\sigma_{q,t}$ of algebra $\Lambda$ defined above by (\ref{auto}):
$$
\sigma_{q,t}(p_{r})=\frac{1-q^r}{1-t^r} p_{r}.
$$
Then using formula (6.19) from \cite{Ma} (see page 327) it is  easy to verify that
\begin{equation}
\label{macok}
 \sigma_{q,t}\left( P_{\lambda}(z,q,t)\right)=(-1)^{|\mu|}\frac{H(\lambda,q,t)}{H(\lambda^{\prime},t,q)}
 {P}_{\lambda'}(z,t,q).
\end{equation}
   Let now
$x=(x_{1},x_{2},\dots), y=(y_{1},y_{2},\dots,)$ be two infinite
sequences. Then we have (see \cite{Ma}, page 345,
formula $(7.9^{'})$)
\begin{equation}
\label{mac} P_\lambda(x,y,q,t)=\sum_{\mu\subset\lambda}
P_{\lambda/\mu}(x,q,t)P_{\mu}(y,q,t),
\end{equation}
where $P_{\lambda/\mu}(z,q,t)$ are the {\it skew Macdonald
functions} defined in \cite{Ma} (see Chapter 6, $\S$ 7) and
$\mu\subset\lambda$ means that $\mu_i \leq \lambda_i$ (or
equivalently the diagram of $\mu$ is a subset of the diagram of
$\lambda$).

If we apply this automorphism $\sigma_{q,t}$ acting in $y$
variables on both sides of the formula (\ref{mac}) and put all
the variables $x$ and $y$ except the first $n$ and $m$
respectively to zero we get
\begin{equation}
\label{Jack} \varphi (P_{\lambda}(z,q,t))=\sum_{\mu\subset\lambda}
(-1)^{\mid\mu\mid}P_{{\lambda} /{\mu}}(x,q,t)\frac{H(\mu,q,t)}
{H(\mu^{\prime},t,q)}P_{\mu^{\prime}}(y,t,q).
\end{equation}

Now let us assume that $\lambda$ is not contained in the fat
$(n,m)$-hook, then $\lambda'_{m+1} > n$. We have two
possibilities: $\mu'_{m+1} > 0$ or $\mu'_{m+1} = 0$. In the first
case we have $P_{\mu'}(y_{1},\dots, y_{m},t,q)=0$. In the
second case we have $\lambda'_{m+1} - \mu'_{m+1} > n$, so
according to \cite{Ma} (page 347, formula (7.15)) the skew
function $P_{{\lambda}/{\mu}}(x_{1},\dots,x_{n},q,t)=0$. Thus we
have shown that the Macdonald polynomials $P_{\lambda}(z,q,t)$
with $\lambda \in \bar H_{n,m}$ belong to the kernel of $\varphi.$

To prove that they actually generate the kernel let us consider
the image of the Macdonald polynomials $P_{\lambda}(z,q,t)$
with $\lambda \in  H_{n,m}.$ From the formula (\ref{Jack}) it
follows that the leading term in lexicographic order of $\varphi
(P_{\lambda}(z,q,t))$ has a form
 $$
(-1)^{\mid\mu\mid}\frac{H(\mu,q,t)}{H(\mu^{\prime},t,q)}{x_{1}}^{{\lambda}_{1}}
\dots {x_{n}}^{{\lambda}_{n}}{y_{1}}^{{\mu}_{1}^{\prime}}\dots
{y_{m}}^{{\mu}_{m}^{\prime}}
 $$
where $\mu =({\lambda}_{n+1}, {\lambda}_{n+2}, \dots )$.  From the
definition $\varphi (P_{\lambda}(z,q,t))\in \Lambda_{n,m,q,t}$. It
is clear that all these polynomials corresponding to the diagrams
contained in the fat hook are linearly independent.  The Theorem 
is proved.
\end{proof}

Note that we have also shown that for the generic $q,t$ the
restriction of the Macdonald polynomials on the subvariety $\Delta_{n,m, q,t} \subset \mathcal M$
\begin{equation}
\label{superJack}
SP_{\lambda}(x,y,q,t)=\varphi(P_{\lambda}(z,q,t)), \quad \lambda
\in H_{n,m}
\end{equation}
form a basis in ${\Lambda}_{n,m,q,t}$. By analogy with Jack polynomials case (see e.g. \cite{SV2}) we call $SP_{\lambda}(x,y,q,t)$ the {\it super Macdonald polynomials.}

\begin{corollary} 
\label{RES}
 Let $f \in \Lambda_{t}$ be a shifted
symmetric function and ${\mathcal M}^f_{q,t} = \chi(f) \in {\mathcal
D}_{q,t}$ be the corresponding difference operator commuting with MR operator ${\mathcal M}_{q,t}$. Then for generic $q,t$ there exists a difference operator
$M^f_{n,m,q,t}$ commuting with the deformed MR operator (\ref{defMR})
such that the following diagram is commutative
$$\begin{array}{ccc} \Lambda& \stackrel{{\mathcal
M}^f_{q,t}}{\longrightarrow}&\Lambda
\\ \downarrow \lefteqn{\varphi}& &\downarrow \lefteqn{\varphi}\\
\Lambda_{n,m,q,t}&\stackrel{{\mathcal M}^{f}_{n,m,q,t}}
{\longrightarrow}&\Lambda_{n,m,q,t} \\
\end{array}
$$ 
All the operators $M^f_{n,m,q,t}$ commute with each other and the super Macdonald polynomials (\ref{superJack}) are their joint
eigenfunctions.
\end{corollary}

{\bf Remark.} For special $q,t$ of the form $q=t^k$ where $k$ is nonnegative integer the
homomorphism $\varphi$ can be passed through the finite dimension
$N = n+mk$:  $\varphi = \phi \circ \varphi_N,$ where $\varphi_N :
\Lambda \rightarrow \Lambda_N$ is the standard map (all $z_i$
except $N$ go to zero), and $\phi: \Lambda_N \rightarrow
\Lambda_{n,m,q,t}$ is a homomorphism such that
$$
\phi(z_{i})=x_{i},\: i=1,\dots,n,\quad\phi(z_{n+kl+j})=t^{j-1}y_{l},\, l=1,\dots,m,\; j=1,\dots, k.
$$
When $m=1$ the
corresponding ideal (the kernel of $\phi$) 
was investigated by B. Feigin, Jimbo, Miwa and Mukhin in
\cite{FJMM}, who also described it in terms of
Macdonald polynomials, but their description is much more
complicated then in generic parameters case. 
The case $k=2, m=2$ was studied in \cite{KMSV}.

We investigate the homomorphism: $f
\rightarrow M^f_{n,m,q,t}$ in more detail in the next section, but here we finish with the following result mentioned at the beginning of this section. 

Recall that the parameters $q,t$ are call special if $q^a = t^b$ for some non-negative integers $a, b$ not equal to zero simultaneously.

\begin{thm} \label{defN}
If the parameters $q,t$ are non-special  then
$\Lambda_{n,m,q,t}$ is generated by the deformed Newton sums
$ p_{r}(x,y,q,t).$
\end{thm}

The main idea is standard for this kind of result (see e.g. \cite{SV}).
We show that the dimensions of the homogeneous components of the algebra $\Lambda_{n,m, q,t}$
and its subalgebra $\mathcal N_{n,m, q,t}$ generated by the deformed Newton sums are the same and thus these two algebras coincide. More precisely, we prove that these dimensions coincide with the number of the corresponding Young diagrams contained in the fat $(n,m)$-hook (see Fig. 1 above).


\begin{lemma} If $q,t$ are not special then
the dimension of the homogeneous component $\Lambda_{n,m; q,t}$ of
degree $N$ is less or equal than the number of  partitions
$\lambda$ of $N$ such that $\lambda_{n+1}\le m$.
\end{lemma}

\begin{proof}
For a given partition $\nu$ consider the set of all different partitions $\hat\nu$, which one can get from $\nu$ by eliminating at most one part of it (or one row in the Young diagram representation).
We have the following obvious formula
\begin{equation}\label{mon}
m_{\nu}(x_{1},x_{2},\dots,x_{n})=\sum_{\hat\nu\cup(a)=\nu}x_{1}^{a}m_{\hat\nu}(x_{2},\dots,x_{n})
\end{equation}
where  $(a)$ denote the row of length $a.$ Any element $f\in \Lambda_{n,m,q,t}$ can be written in the form 
$$
f=\sum_{\lambda,\mu}c(\lambda,\mu)m_{\lambda}(x_{1},x_{2},\dots,x_{n})m_{\mu}(y_{1},y_{2},\dots,y_{n})
$$
Then from the equality (\ref{mon}) we have the linear system
\begin{equation}\label{sys}
\sum_{a+b=p}(q^{a}-t^b)c(\lambda\cup(a),\mu\cup(b))=0,
\end{equation}
where $a,b,p$ nonnegative integers and $\lambda,\mu$ are partitions such that $\lambda_{n}=0,\:\mu_{m}=0$.

Consider the set  $X_{N}(n,m)$ of pairs  $\lambda,\mu$ of partitions such that$ |\lambda|+|\mu|=N,\:\lambda_{n+1}=0,\:\mu_{m+1}=0$. Then we have disjoint union
$$
X_{N}(n,m)=D_{N}(n,m)\cup  R_{N}(n,m),
$$
where $D_{N}(n,m)$ is the subset  of pairs $(\lambda,\mu)$ of partitions such that $\lambda_{n+1}\le m.$ 
Note that the corresponding Young diagrams $\lambda\cup\mu'$, where $\mu'$ corresponds to the transposed Young diagram, are precisely those which are contained in the fat $(n,m)$-hook. 
We would like to show that one can express $c(\lambda,\mu)$ for all others partitions as a linear combinations of those from $D_{N}(n,m).$

Consider any $(\lambda,\mu)\in  R_{N}(n,m): \lambda = (\lambda_1, \dots, \lambda_n), \, \mu = (\mu_1, \dots, \mu_s).$ Define 
$$
k=\sharp\{i\:|\:\lambda_{i} < \mu^{\prime}_{1}\}
$$ 
and $$\mu(q)=\mu_{s-q+1}+\dots+\mu_{s}$$
for any integer $0 < q \le s.$
Introduce the following partial order on  $R_{N}(n,m)\::
(\lambda,\mu)\prec(\tilde\lambda,\tilde\mu)$ if and only if 
$k<\tilde k$ or one of the following conditions is fulfilled:

$k=\tilde k$ and for $q=min\{\lambda_n,\tilde \lambda_n\},\:\mu(q+1) <\tilde\mu(q+1)$ 
 
 $k=\tilde k$ and for $q=min\{\lambda_n,\tilde \lambda_n\},\:\mu(q+1) =\tilde\mu(q+1)$ and $\lambda_n < \tilde \lambda_n$.
We prove Lemma by induction in $N(x)=\sharp\{y\prec x\mid y\in R_{N}(n,m)\}$.
Let $(\lambda,\mu)\in R_{N}(n,m)$. Consider equation (\ref{sys}) where $\hat\lambda=(\lambda_1, \dots, \lambda_{n-1})$ and     $\hat\mu=(\mu_1, \dots, \mu_{s-\lambda_{n}-1},\mu_{s-\lambda_{n}+1},\dots,\mu_{s})$ and $p=\lambda_{n}+\mu_{s-\lambda_{n}}$. It is easy to see that the equation contains $c(\lambda,\mu)$ with non-zero coefficient and $(\lambda,\mu)$ is maximal among all irregular pairs in this equation.
\end{proof}

Now let us prove the Theorem. To show that ${\mathcal
N}_{n,m,q,t}= \Lambda_{n,m,q,t}$ it is enough to prove that the
dimension of the homogeneous component of degree $N$ of ${\mathcal
N}_{n,m,q,t}$ is not less than $D_{N}(n,m)$. 
To produce enough
independent polynomials consider the super Macdonald polynomials (\ref{superJack})
$$SP_{\lambda}(x,y,q,t) = \varphi (P_{\lambda}(z,q,t)),  \quad \lambda
\in H_{n,m}.$$
From the formula (\ref{Jack}) below it follows that the leading term of $SP_{\lambda}(x,y,q,t))$ in
lexicographic order has a form
 $$
  x_{1}^{\lambda_{1}}\dots
x_{n}^{\lambda_{n}}y_{1}^{<{\lambda}^{\prime}_{1}-n>}\dots
  y_{m}^{<{\lambda}^{\prime}_{m}-n>},
 $$
where $\lambda^{\prime}$ is the partition conjugate to $\lambda$
and $<x>=\frac{x+|x|}{2} = max(0,x)$. From the definition $\varphi
(P_{\lambda}(z,q,t))\in {\mathcal N}_{n,m,q,t}$. It is clear that
all these polynomials corresponding to the diagrams contained in
the fat hook are linearly independent. This completes the proof of
Theorem \ref{defN}.

\section{Shifted super Macdonald polynomials and Harish-Chandra homomorphism}

Let again $P_{n,m}={\mathbb
C}[x_{1},\dots,x_{n},y_{1},\dots,y_{m}]$ be polynomial algebra in
$n+m$ independent variables. The following algebra $\Lambda_{n,m,q,t}^{\natural}$ can be considered as a shifted version of the algebra $\Lambda_{n,m,q,t}.$ It consists of the polynomials $p(x_1,\dots,x_n, y_1, \dots, y_m),$ which are symmetric in
$x_{1},x_{2}t,\dots,x_{n}t^{n-1}$ and \\$y_{1},y_{2}q\dots,y_{m}q^{m-1}$ separately
and satisfy the conditions
\begin{equation}
\label{condshift} T_{q,x_{i}}(f)=T_{t,y_{j}}(f)
\end{equation}
 on each
hyperplane  $x_{i}t^{i-1}- y_{j}q^{j-1}=0$ for $i=1,\dots,n$ and $j=1,\dots,m.$

Now we are going to define the homomorphism $\varphi^{\natural}$
which is a shifted version of the homomorphism $\varphi$ from the
previous section.

Recall that $H_{n,m}$ denote the set of partitions $\lambda$ whose
diagrams are contained in the fat $(n,m)$-hook. Consider the
following  $ F : H_{n,m} \longrightarrow \mathbb{C}^{n+m}$ : $F
(\lambda)=(p_{1},\dots,p_{n},q_{1},\dots,q_{m}),$ where
$$p_{i}=
q^{\lambda_{i}},\quad q_{j}=t^{\mu_{j}^{'}}t^{n},
$$
and $\mu=(\lambda_{n+1},\lambda_{n+2},\dots).$ The image
$F(H_{n,m})$ is dense in $\mathbb{C}^{n+m}$ with respect to
Zariski topology. The homomorphism
$$ \varphi^{\natural}:
\Lambda_{t}\longrightarrow {\bf
C}[x_{1},\dots,x_{n},y_{1},\dots,y_{m}] $$ is defined by the
relation $$ \varphi^{\natural}(f)(p,q)=f((F^{-1}(p,q)),$$ where
$(p,q)\in F(H_{n,m})$. In other words we consider the shifted
symmetric function $f$ as a function on the partitions from the
fat hook and re-write it in the new coordinates. The fact that as
a result we will have a polynomial is not obvious.

\begin{lemma}\label{4}
 The image $\varphi^{\natural}(f)$ of a shifted
symmetric function $f \in \Lambda_{t}$ is a polynomial. For the
shifted power sums  $p^{*}_{k}(z,t)$ it can be given by the
following explicit formula:
\begin{eqnarray}
\label{frob} \varphi^{\natural}\left(p_{k}^{*}(z,t)\right)=
\sum_{i=1}^n(x_{i}^r-1)t^{r(i-1)}+\frac{1-q^r}{1-t^r}\sum_{j=1}^m
(y_{j}^r-t^{rn})q^{r(j-1)}.
\end{eqnarray}
\end{lemma}

\begin{proof} Assume that $z_{i}=\lambda_{i}$, where $\lambda\in
H_{n,m}$. Then we have
$$
\varphi^{\natural}(p_{k}^{*}(z,t))=
\sum_{i\ge
1}(q^{r\lambda_{i}}-1)t^{i-1}=
\sum_{i=1}^n(q^{r\lambda_{i}}-1)t^{r(i-1)}+t^{rn}\sum_{i\ge
1}(q^{r\lambda_{n+i}}-1)t^{r(i-1)}.
$$
Now using (\ref{conjugate})  we have
$$
\sum_{i\ge
1}(q^{r\lambda_{n+i}}-1)t^{r(i-1)}=\frac{1-q^r}{1-t^r}\sum_{j=1}^m
(t^{r\mu_{j}^{\prime}}-1)q^{r(j-1)},
$$
which proves the formula (\ref{frob}). Since the shifted sums
generate $\Lambda_{t}$ this implies the first part of Lemma as
well.
\end{proof}

\begin{thm}\label{hook}  If the parameters $q,t$ are generic then the
image of the homomorphism $\varphi^{\natural}$ coincides with the
algebra ${\Lambda}_{n,m,q,t}^{\natural}$ and the kernel of
$\varphi^{\natural}$ is spanned by the shifted Macdonald
polynomials $ P_{\lambda}^{*}(z,q,t)$ corresponding to the Young
diagrams which are not contained in the fat $(n,m)$-hook.
\end{thm}

\begin{proof}
The first claim follows from Lemma \ref{4} and Theorem \ref{defN}.
To prove the statement about the kernel consider a
shifted Macdonald polynomial $P^*_{\lambda}(z,q,t)$ with $\lambda
\in \bar H_{n,m}.$ Let $\mu$ be a partition whose diagram is
contained in the fat $(n,m)$-hook. Since this implies that the
diagram of $\lambda$ is not a subset of the one of $\mu$ according
to the Extra Vanishing Property of shifted Macdonald polynomials
(see Section 3) we have $P^*_{\lambda}(q^{\mu},q,t)=0.$ Thus we have
shown that $P^*_{\lambda}(z,q,t)$ with $\lambda \in \bar H_{n,m}$
belong to the kernel of $\varphi.$ To show that they generate the
kernel one should note that $$
 \varphi^{\natural}(P^*_{\lambda}(z,q,t))=\varphi (P_{\lambda}(z,q,t))
 (x_{1},x_{2}t,\dots,x_{n}t^{n-1}, y_{1},y_{2}q\dots,
 y_{m}q^{m-1})+\dots,
$$
where dots mean the terms of degree less then $|\lambda|$. From
theorem \ref{ker} it follows that
$\varphi^{\natural}(P^*_{\lambda}(z,q,t))$ with $\lambda\in
H_{n,m}$ are linearly independent. Theorem is proved.
\end{proof}

\begin{corollary} 
 For generic $q,t$ the functions
$$SP^*_{\lambda}(x,y,q,t)=\varphi^{\natural}(P^*_{\lambda}(z,q,t))$$
with $\lambda\in H_{n,m}$ form a basis in ${\Lambda}_{n,m,q,t}^{\natural}$.
\end{corollary} 

We will call the polynomials $SP^*_{\lambda}(x,y,q,t)$ the {\it
shifted super Macdonald polynomials.} We are going to show that to any such polynomial
corresponds a quantum integral of the deformed MR system.

Let us consider the algebra of difference operators in $n+m$
variables with rational coefficients belonging to ${\mathbb
C}[x_{1},\dots,x_{n},y_{1},\dots,y_{m}, (x_{i}-x_{j})^{-1},
(x_{i}-y_{l})^{-1},  (y_{k}-x_{l})^{-1})], \quad 1\le i<j\le n,
1\le l < k \le m.$ We denote it as $\Delta(n,m).$

\begin{thm}  For generic values of $q,t$ there exists a
unique monomorphism $\psi: \Lambda_{n,m,q,t}^{\natural} \rightarrow
\Delta(n,m)$ such that the following diagram is commutative
$$
\begin{array}{ccc}
\Lambda_{t}&\stackrel{\chi}{\longrightarrow}&\Delta(q,t)
\\ \downarrow \lefteqn{\varphi^{\natural}}& &\downarrow \lefteqn{res}\\
\Lambda^{\natural}_{n,m,q,t}&\stackrel{\psi}{\longrightarrow}&\Delta(n,m),
\\
\end{array}
$$
where $\chi$ is the inverse Harish-Chandra homomorphism 
and $res$ is the operation of restriction onto $\Delta(n,m,q,t)$ described by Corollary \ref{RES}.
\end{thm}

Indeed let $f$ be a shifted symmetric function from $\Lambda_{t},$
${\mathcal M}^{f}_{q,t}$ and ${\mathcal M}^{f}_{n,m,q,t} = res
({\mathcal M}^{f}_{q,t})$ be the same as in Corollary \ref{RES}. We know that if
$P_{\lambda}(z, q,t)$ is a Macdonald symmetric function then
$${\mathcal M}^{f}_{n,m,q,t} \varphi(P_{\lambda}(z,
q,t))=f(q^{\lambda})\varphi(P_{\lambda}(z, q,t)).$$ Therefore
according to Theorem \ref{ker}
${\mathcal M}^{f}_{n,m,q,t}\equiv 0$
if and only if $f(q^\lambda)=0$ for any $\lambda$ with the diagram
contained in the fat $(n,m)$-hook. Now from Theorem \ref{hook} it follows
that $Ker(res \circ \chi) = Ker \varphi^{\natural}$.

\section{Combinatorial formulas}

In this section we give some combinatorial formulas for the
super Macdonald polynomials and shifted super Macdonald
polynomials generalising the results by  Okounkov \cite{Ok}. 
Let us recall his results.

A tableau $T$ on $\lambda$ is called a {\it reverse tableau} if
its entries
 decrease strictly downwards in each column and weakly rightwards in each row. 
By $T(s)$ we denote the entry in the box
$s\in\lambda$. The following combinatorial formula for the shifted
Macdonald polynomial was proven by Okounkov in \cite {Ok}:
\begin{equation}
\label{shiftedjack} P_{\lambda}^{*}(x_1,\dots, x_N,q,t)=\sum_{T}
\psi_{T}(q,t)\prod_{s\in\lambda}
\left(x_{T(s)}-q^{a^{\prime}(s)}t^{l^{\prime}(s)}\right)t^{T(s)-1}
\end{equation}
where $a^{\prime}(s)$ and $l^{\prime}(s)$ are defined for a box $s = (i,j)$ as
$$a^{\prime}(s) = j-1,\quad l^{\prime}(s)= i-1.$$ 
Here the sum is taken over all reverse tableaux on $\lambda$
with entries  in $\{1,\dots,N\}$ and $\psi_{T}(q,t)$ is the same
weight as in the combinatorial formulas for the ordinary
Macdonald polynomials (see  \cite{Ma}, VI, (7.13')) interpreted in terms of reverse tableau:
\begin{equation}
\label{jack2} P_{\lambda}(x_1,\dots, x_N,q,t)=\sum_{T}
\psi_{T}(q,t)\prod_{s\in\lambda} x_{T(s)}.
\end{equation}
\begin{equation}
\label{jack3} P_{\lambda/\mu}(x_1, \dots, x_N,q,t)=\sum_{T}
\psi_{T}(q,t)\prod_{s\in\lambda/\mu} x_{T(s)}.
\end{equation}
In the last formula the sum is taken over all reverse tableaux of shape
$\lambda/\mu$ with entries in $\{1,\dots,N \}.$

Let us consider now a reverse {\it bitableau} $T$   of type
$(n,m)$ and shape $\lambda$. We  can view $T$ as a filling of a
Young diagram  $\lambda$ by symbols $1<2<\dots
<n<1^{\prime}<2^{\prime}<\dots< m^{\prime}$ with entries
 decreasing weakly downwards in each column and rightwards in each row; additionally 
entries $1,2\dots ,n$ decrease strictly  downwards in each column and
entries $1^{\prime},2^{\prime}\dots, m^{\prime}$ decrease strictly
rightwards in each row. Let $T_{1}$ be a subtableau  in $T$ containing
all symbols $1^{\prime},2^{\prime}\dots, m^{\prime}$ and
$T_{0}=T-T_{1}$ and $\mu$ is the shape of $T_{1}$.

\begin{thm} 
 For generic values of the parameters $q,t$
the super Macdonald  polynomials can be written as
\begin{equation}
\label{deformedjack}
SP_{\lambda}(x_{1},x_{2},\dots,x_{n},y_{1},y_{2},\dots,y_{m},q,t)=\sum_{T}
\psi_{T}(q,t)\prod_{s\in\lambda}x_{T(s)}
\end{equation}
where $x_{j^{\prime}}$ is denoted as $y_{j}$ and
 $$
\psi_{T}(q,t)=(-1)^{|\mu|}
\psi_{T_{1}^{\prime}}(t,q)\psi_{T_{0}}(q,t)\frac{H(\mu,q,t)}
{H(\mu^{\prime},t,q)} $$
\end{thm}

Proof follows directly from the formulas
(\ref{Jack}),(\ref{jack2}),(\ref{jack3}).

We are going to present now a combinatorial formula for
the shifted super Macdonald polynomial.

\begin{thm}  The following formula is true:
 $$
SP^{*} _{\lambda} =\sum_{T} \psi_{T}(q,t)\prod_{s\in\lambda}
\left(x_{T(s)}-q^{a^{\prime}(s)}t^{l^{\prime}(s)}\right)(q,t;s)^{T(s)-1}
$$
where $(q,t;s)=q$ if $s\in T_{1}$ and $(q,t;s)=t$ if $s\in
T_{0}$.
\end{thm}

\begin{proof} Let us consider the skew diagram $\lambda/\mu$ in the
formula (\ref{jack2}) and define $$
P_{\lambda/\mu}^{*}(x,q,t)=\sum_{T}
\psi_{T}(q,t)\prod_{s\in\lambda/\mu}
\left(x_{T(s)}-q^{a^{\prime}(s)}t^{l^{\prime}(s)}\right)t^{T(s)-1}
.$$

Okounkov in \cite{Ok} proved the following formula
\begin{equation}
\label{okun}
P_{\lambda}^{*}(z_{1},z_{2}\dots,q,t)=\sum_{\mu\prec\lambda}\psi_{\lambda/\mu}
(q,t)\prod_{s\in\lambda/\mu}
\left(z_{1}-q^{a^{\prime}(s)}t^{l^{\prime}(s)}\right)t^{|\mu|}
P_{\mu}^{*}(z_{2},z_{3}\dots,q,t),
\end{equation}
 where $\mu\prec\lambda$ means
$\lambda_{i+1}\le\mu_{i}\le\lambda_{i}$ and
$\psi_{\lambda/\mu}(q,t)$ is the same coefficient  as for
the ordinary Macdonald polynomials \cite{Ma}
$$
P_{\lambda}(z_{1},z_{2}\dots,q,t)=\sum_{\mu\prec\lambda}
\psi_{\lambda/\mu}(q,t)
z_{1}^{|\lambda/\mu|}P_{\mu}(z_{2},z_{3}\dots,q,t).
$$
Strictly speaking Okounkov proved this for finitely many variables. To make sense of formula (\ref{okun})  in infinite dimension we embed the algebra $\Lambda_t$ to $\mathbb C[z_1] \otimes \Lambda_t$ by sending
$p_k^*$ to $z_1^{k}-1 + t^{k} p_k^*.$

Applying Okounkov's formula $n$ times we get $$
P_{\lambda}^{*}(z_{1},z_{2}\dots,q,t)=\sum_{\mu\subset\lambda}
P_{\lambda/\mu}^{*}(z_{1},z_{2},\dots,z_{n},q,t)t^{n|\mu|}
P_{\mu}^{*}(z_{n+1},z_{n+2}\dots,q,t),
$$ which implies $$
\varphi^{\natural}(P_{\lambda}^{*}(z_{1},z_{2}\dots,q,t))=
\sum_{\mu\subset\lambda}
P_{\lambda/\mu}^{*}(x_{1},x_{2},\dots,x_{n},q,t)t^{n|\mu|}\omega^*_{q,t}
(P_{\mu}^{*}(z_{n+1},z_{n+2}\dots,q,t)).
$$
Now using the duality (\ref{dual}) we have
$$
\varphi^{\natural}(P_{\lambda}^{*}(z_{1},z_{2}\dots,q,t))=
\sum_{\mu\subset\lambda}P_{\lambda/\mu}^{*}(x_{1},x_{2},\dots,x_{n},q,t)
\frac{H(\mu,q,t)}{H(\mu^{\prime},t,q)}
P_{\mu^{\prime}}^{*}(y_{1},y_{2}\dots,y_{m},t,q).
$$ But according to formula (\ref{shiftedjack})
$$
P_{\mu^{\prime}}^{*}(y_{1},y_{2}\dots,y_{m},t,q)=\sum_{T^{\prime}_{1}}
\psi_{T^{\prime}_{1}}(t,q)\prod_{s^{\prime}\in\mu^{\prime}}
\left(x_{T^{\prime}_{1}(s^{\prime})}-t^{a^{\prime}(s^{\prime})}q^{l^{\prime}
(s^{\prime})} \right)q^{T_{1}^{\prime}(s^{\prime})-1}
$$
$$ =\sum_{T_{1}} \psi_{T_{1}^{\prime}}(t,q)\prod_{s\in\mu}
\left(x_{T_{1}}(s)-q^{a^{\prime}(s)}t^{l^{\prime}(s)} \right
)q^{T_{1}(s)-1}.
$$ Therefore
$$ \varphi^{\natural}
(P_{\lambda}^{*}(z_{1},z_{2}\dots,q,t))=\sum_{T}t^{|\mu|}
\frac{H(\mu,q,t)}{H(\mu^{\prime},t,q)}
\psi_{T}(q,t)\prod_{s\in\lambda/\mu}\left(x_{T}(s)-
q^{a^{\prime}(s)}t^{l^{\prime}(s)}\right )t^{{T(s)}-1}
$$
$$
\psi_{T_{1}^{\prime}}(t,q)\prod_{s\in\mu} \left(x_{T_{1}}(s)-
q^{a^{\prime}(s)}t^{l^{\prime}(s)}\right )q^{{T_{1}(s)}-1}
 $$
and the Theorem is proved.
 \end{proof}

\section{Concluding remarks}

 Haglund, Haiman and Loehr \cite{HHL} recently proved a remarkable new combinatorial formula for Macdonald polynomials, previously proposed by Haglund. These formulas are much more effective than the original Macdonald's formulas (\ref{jack2}) and provide a new direct way to prove the main properties of Macdonald polynomials.
A natural problem is to find an analogue of Haglund's formula for the super Macdonald polynomials.

Another open problem is to find generalisations of our results for the deformed analogues of the Macdonald operators related to other Lie superalgebras (see \cite{SV}), in particular for the deformed Koornwinder operators. We hope to address some of these questions in the nearest future.

\section{Acknowledgments}

We are grateful to  A. Okounkov for useful and stimulating discussions. 

This work has been partially supported by the EPSRC (grant EP/E004008/1),
European Union through the FP6 Marie Curie RTN ENIGMA (Contract
number MRTN-CT-2004-5652) and ESF programme MISGAM.


\begin{thebibliography}{99}
\bibitem{SV}
A.N. Sergeev, A.P. Veselov {\it Deformed quantum
Calogero-Moser Problems and Lie superalgebras.} 
 Comm. Math. Phys.  245  (2004),  no. 2, 249--278.
\bibitem{SV2}
A.N.  Sergeev, A.P. Veselov {\it  Generalised discriminants, deformed Calogero-Moser-Sutherland operators and super-Jack polynomials.} Adv. Math. 192 (2005), no. 2, 341--375. 
\bibitem{R}
S.N.M. Ruijsenaars {\it Complete integrability of relativistic Calogero-Moser systems and elliptic function identities.} Comm. Math. Phys. 110 (1987), no. 2, 191--213.
\bibitem{Ma}
I. Macdonald {\it Symmetric functions and Hall polynomials} 2nd
edition, Oxford Univ. Press, 1995.
\bibitem{KSa}
F. Knop, S. Sahi, {\it Difference equations and symmetric
polynomials defined by their zeros.} Internat. Math. Res. Notes
{\bf 10}, 1996, 437-486.
\bibitem{K}
F. Knop {\it Symmetric and non-symmetric quantum Capelli polynomials.} Comment. Math. Helv. 72 (1997), no. 1, 84--100.
\bibitem{Sa}
S. Sahi {\it Interpolation, integrality, and a generalization of Macdonald's polynomials}.  Internat. Math. Res. Notices  (1996),  no. 10, 457--471. 
\bibitem{Ok}
A. Okounkov {\it (Shifted) Macdonald polynomials : $q$- integral
representation and combinatorial formula.} Compositio Math. {\bf
112}, 1998, no.2, 147-182.
\bibitem{FJMM}
B. Feigin, M. Jimbo, T. Miwa, E., Mukhin.{\it Symmetric polynomials vanishing on the shifted diagonals and Macdonald polynomials.}  Int. Math. Res. Not.  2003,  no. 18, 1015--1034.
\bibitem{Cha}
O.A. Chalykh {\it Macdonald polynomials and algebraic integrability}.  Adv. Math.  166  (2002),  no. 2, 193--259.
\bibitem{CFV}
A.P. Veselov, M.V. Feigin, O.A. Chalykh {\it New integrable
deformations of quantum Calogero - Moser problem.} Russian Math.
Surveys {\bf 51},  no.3, 1996, 185--186.
\bibitem{Ch}
I. Cherednik, Double affine Hecke algebras and Macdonald's conjectures, Ann. of Math. (2) 141 (1995), no. 1, 191-216.
\bibitem{KN}
A.N. Kirillov, M. Noumi {\it Affine Hecke algebras and raising operators for Macdonald polynomials}.  Duke Math. J.  93  (1998),  no. 1, 1--39.
\bibitem{KMSV}
M. Kasatani, T. Miwa, A.N. Sergeev, A.P. Veselov {\it Coincident root loci and Jack and Macdonald polynomials for special values of the parameters.} Jack, Hall-Littlewood and Macdonald polynomials, 207--225, Contemp. Math., 417, Amer. Math. Soc., Providence, RI, 2006. 
\bibitem{Atiyah-Macdonald}
M. Atiyah, I.G. Macdonald {\it Introduction to Commutative
Algebra.} Addison-Wesley, 1969.


\bibitem{HHL}
J. Haglund, M. Haiman, N. Loehr {\it A combinatorial formula for Macdonald polynomials.}  J. Amer. Math. Soc.  18  (2005),  no. 3, 735--761. 





\end{thebibliography}
\end{document}